\documentclass[reqno]{amsart}
\usepackage{amssymb}
\usepackage[mathscr]{eucal}
\usepackage{hyperref}

\theoremstyle{plain}
\newtheorem{prop}{Proposition}[section]
\newtheorem{thm}[prop]{Theorem}
\newtheorem{cor}[prop]{Corollary}
\newtheorem{lem}[prop]{Lemma}

\theoremstyle{definition}
\newtheorem{dfn}[prop]{Definition}
\newtheorem{rem}[prop]{Remark}
\newtheorem{rems}[prop]{Remarks}
\newtheorem{example}[prop]{Example}
\newtheorem{lab}[prop]{}

\numberwithin{equation}{section}

\newcommand{\into}{\hookrightarrow}

\renewcommand{\subset}{\subseteq}

\newcommand{\C}{{\mathbb{C}}}

\renewcommand{\P}{{\mathbb{P}}}
\newcommand{\Q}{{\mathbb{Q}}}
\newcommand{\R}{{\mathbb{R}}}
\newcommand{\Z}{{\mathbb{Z}}}

\DeclareMathOperator{\Gal}{Gal}
\DeclareMathOperator{\Hom}{Hom}
\DeclareMathOperator{\tr}{tr}

\newcommand{\x}{{\mathtt{x}}}

\renewcommand{\epsilon}{\varepsilon}
\newcommand{\ol}{\overline}
\newcommand{\ex}{\exists\,}

\renewcommand{\choose}[2]{\genfrac(){0pt}{}{#1}{#2}}
\newcommand\http[1]{\href{http://#1}{\nolinkurl{#1}}}
\newenvironment{fitbigdiagram}{\begingroup\small}{\endgroup}
\newcommand{\Label}[1]{\label{#1}}

\begin{document}

\title
[Sums of squares of polynomials with rational coefficients]
{Sums of squares of polynomials\\ with rational coefficients}

\author{Claus Scheiderer}
\address
  {Fachbereich Mathematik and Statistik \\
  Universit\"at Konstanz \\
  D--78457 Konstanz \\
  Germany}
\email{claus.scheiderer@uni.constanz.de}
\urladdr{http://www.math.uni-konstanz/\textasciitilde scheider}

\begin{abstract}
We construct families of explicit polynomials $f$ over $\Q$ that are
sums of squares of polynomials over $\R$, but not over $\Q$. Whether
or not such examples exist was an open question originally raised by
Sturmfels. We also study representations of $f$ as sums of squares of
rational functions over $\Q$. In the case of ternary quartics, we
prove that our counterexamples to Sturmfels' question are the only
ones.
\end{abstract}

\date\today
\maketitle


\section*{Introduction}

Let $f(x_1,\dots,x_n)$ be a polynomial with rational coefficients,
and assume that $f$ is a sum of squares of polynomials with real
coefficients. A few years ago, Sturmfels raised the question whether
$f$ is necessarily a sum of squares of polynomials with rational
coefficients. The main result of this paper gives a negative answer
to this question.

The background for this question comes from semidefinite programming
(see e.g.\ \cite{wsv},
\cite{bv},
\cite{n},
\cite{al})
and more specifically, from polynomial optimization. Lasserre's
method of moment relaxation \cite{la}
gives, in principle, positivity certificates for real polynomials
based on sums of squares decompositions. However, even if the initial
data is exact, e.g.\ given by polynomials with rational coefficients,
the algorithm produces floating point solutions, and therefore the
output is not necessarily reliable. One would like to understand to
what extent one can expect exact certificates, see for instance
\cite{sz},
\cite{klyz}.
The question by Sturmfels addresses this issue in its most basic
form.

From general reasons, it is clear that $f$ has a sum of squares
representation over some real number field~$K$. So far, it was known
by work of Hillar \cite{hi} that the question has a positive answer
when $K$ is totally real. Under this assumption, Hillar also gave a
bound for the number of squares needed over $\Q$, in terms of the
number needed over $K$ and of the degree of the Galois closure of $K$
over $\Q$. Quarez \cite{qu} later gave a different proof to the same
result and improved Hillar's bound significantly. Both proofs are
constructive. In Section \ref{sect:totreal} we revisit the result
and show that it is essentially an immediate consequence of
well-known properties of the trace form of $K/\Q$. Our argument is
constructive as well. In addition it gives various new information,
for instance that the bound found by Quarez holds in the non-Galois
as well.

In Section \ref{sect:stq} we present explicit counterexamples to the
question by Sturmfels. Working with homogeneous polynomials (forms)
we construct, for any integer $n\ge2$ and any even number $d\ge4$, a
family of forms $f\in\Q[x_0,\dots,x_n]$ of degree~$d$ that are sums
of two squares of forms over $\R$, but not sums of squares of forms
over~$\Q$ (Theorem \ref{stq}). These forms $f$ are the $K/\Q$-norms
of linear forms defined over suitable number fields $K$ of
degree~$d$.
As a by-product, we show for any real number
field $k$ that there is no analogue of Hilbert's theorem on
nonnegative ternary quartics (the qualitative part): There always
exists a nonnegative ternary quartic form with coefficients in $\Q$
that is not a sum of squares of forms over~$k$ (Corollary
\ref{stqgend}).

Any nonnegative form $f\in\Q[x_0,\dots,x_n]$ is a sum of squares of
rational functions over $\Q$, according to Artin. In Section
\ref{sect:denoms} we study such representations for the family of
counterexamples constructed in Section \ref{sect:stq}. If $f$ is such
a form with $\deg(f)=d$, we prove (Theorem \ref{denomh17}) that
there always exists a nonzero form $h$ over $\Q$ of degree $d-2$, but
not of any smaller degree, for which $fh$ is a sum of squares over
$\Q$. In fact, we explicitly construct all such forms $h$
(Proposition \ref{beautiful}). For $d=4$, this yields in particular
an explicit representation of $f$ as a sum of squares of rational
functions \`a~la Artin.

In Section \ref{sect:tq} we prove a partial converse to the
construction from Section \ref{sect:stq}. In the case $(n,d)=(2,4)$
of ternary quartics, we show that every counterexample to Sturmfels'
question arises from our construction (Theorem \ref{terquaronly}).
The proof makes use of a canonical linear subspace $U_f\subset\R[x_0,
\dots,x_n]$ associated with any sum of squares~$f\in\R[x_0,\dots,
x_n]$. We call $U_f$ the characteristic subspace associated with~$f$.
This notion is useful in other situations as well.
At the end of the paper we collect a few open questions.

I would like to thank Marie-Fran\c coise Roy and Ronan Quarez for
stimulating discussions. In particular, the results of Section
\ref{sect:denoms} were prompted by a question of Roy.


\section{Descending sums of squares representations\\ in totally real
  extensions}\Label{sect:totreal}%

Let $f\in\Q[x_1,\dots,x_n]=\Q[\x]$ be a polynomial, and assume that
$f$ is a sum of squares of polynomials in $K[\x]$ where $K$ is a real
number field. In this section we review the result of Hillar
\cite{hi} according to which $f$ is a sum of squares in $\Q[\x]$. We
will show that it is a simple consequence of properties of the trace
form of $K/\Q$. As a consequence, we will generalize the bound of
Quarez \cite{qu} to the case where $K/\Q$ is not necessarily Galois.

\begin{lab}\Label{tracerapp}%
Before giving the actual proof, which is very short, we need to
recall a few facts about trace quadratic forms. Let $K/k$ be a finite
separable field extension of degree $d:=[K:k]$, and consider the
quadratic form
$$\tau\colon K\to k,\quad y\mapsto\tr_{K/k}(y^2)$$
over $k$, where $\tr_{K/k}$ denotes the trace of $K$ over $k$. The
trace form $\tau$ has the following well-known property: For any
ordering $P$ of $k$, the Sylvester signature of $\tau$ with respect
to $P$ is equal to the number of extensions of the ordering $P$
to~$K$. See \cite{s}, Lemma 3.2.7 or Theorem 3.4.5.

Assume that $k$ is real
and that every ordering of $k$ has $d=[K:k]$ different extensions
to $K$, or equivalently, that every ordering of $k$ extends to the
Galois hull of $K$ over~$k$.
Then $\tau$ is positive definite with respect to every ordering of
$k$. Diagonalizing $\tau$ therefore gives sums of squares $a_1,\dots,
a_d$ in $k^*$, together with a $k$-linear basis $y_1,\dots,y_d$ of
$K$, such that
\begin{equation}\Label{trkp}%
\tr_{K/k}\Bigl(\Bigl(\sum_{i=1}^dx_iy_i\Bigr)^2\Bigr)\>=\>\sum_
{i=1}^da_ix_i^2
\end{equation}
holds for all $x_1,\dots,x_d\in k$. Note that we can choose $a_1=d$
here by starting the diagonalization with $y_1=1$.
More generally, if $A$ is an arbitrary (commutative) $k$-algebra and
$A_K=A\otimes_kK$, then
\begin{equation}\Label{tralg}%
\tr_{A_K/A}\Bigl(\Bigl(\sum_{i=1}^dx_i\otimes y_i\Bigr)^2\Bigr)\>=
\>\sum_{i=1}^da_ix_i^2
\end{equation}
holds for all $x_1,\dots,x_n\in A$.
\end{lab}

The following theorem is now a simple observation. It sharpens the
results of Hillar \cite{hi} and Quarez \cite{qu}:

\begin{thm}\Label{improvbound}%
Let $K/k$ be an extension of real fields of degree $d=[K:k]<\infty$,
and assume that every ordering of $k$ extends to $d$ different
orderings of $K$. Then there exist sums of squares $c_1,\dots,c_d$ in
$k$ with $c_1=1$ and with the following property:

For every $k$-algebra $A$, every $m\ge1$ and every $f\in A$ which is
a sum of $m$ squares in $A_K=A\otimes_kK$, there exist $f_1,\dots,f_d
\in A$ such that each $f_i$ is a sum of $m$ squares in $A$, and such
that
$$f\>=\>\sum_{i=1}^dc_if_i.$$
In particular, $f$ is a sum of $dm\cdot p(k)$ squares in $A$. (This
number can be improved, see Remarks \ref{boundsrem} below.)
\end{thm}

Here $p(k)$ denotes the Pythagoras number of $k$, i.\,e., the
smallest number $p$ such that every sum of squares in $k$ is a sum of
$p$ squares in $k$. (If no such number $p$ exists one puts $p(k)=
\infty$.)

\begin{proof}
Choose sums of squares $a_i$ in $k$ and elements $y_i\in K$ ($i=
1,\dots,d$) as in \ref{tracerapp}. It suffices to take $c_i=
\frac{a_i}d$ for $i=1,\dots,d$. Indeed, assuming $f=g_1^2+\cdots+
g_m^2$ with $g_1,\dots,g_m\in A_K$, we get
$$d\cdot f=\tr_{A_K/A}(f)=\sum_{j=1}^m\tr_{A_K/A}(g_j^2)=\sum_{j=1}^m
\sum_{i=1}^da_ix_{ij}^2,$$
where the $x_{ij}\in A$ are determined by $g_j=\sum_{i=1}^dx_{ij}
\otimes y_i$ ($j=1,\dots,m$). So the assertion in the theorem holds
with $f_i=\sum_{j=1}^mx_{ij}^2$ ($i=1,\dots,d$).
\end{proof}

\begin{rems}\Label{boundsrem}%
\smallskip

1.\
The proof is completely constructive: Knowing the sums of squares
decomposition of $f$ in $A_K$, we explicitly get $f_1,\dots,f_d$
together with sums of squares decompositions in~$A$.
\smallskip

2.\
Assume that $k$ is a number field, so $p(k)=4$. Using the well-known
composition formulas for sums of four squares, we can improve the
upper bound $4dm$ in Theorem \ref{improvbound}. Indeed, $c_if_i$ is a
sum of $4\lceil\frac m4\rceil$ squares for every $i$, and is a sum of
$m$ squares for $i=1$, so altogether $f$ is a sum of
$$m+4(d-1)\cdot\left\lceil\frac m4\right\rceil$$
squares in $A$. This is precisely the bound found by Quarez \cite{qu}
in the case where $k=\Q$ and $K/\Q$ is Galois. Note that this bound
lies between $dm$ and $d(m+3)-3$.%
\smallskip

3.\
Similar as in the previous remark, we can improve the bound in
Theorem \ref{improvbound} for arbitrary $K/k$, using composition. In
this way we obtain the general bound
\begin{equation}\Label{genbd}%
8d\cdot\left\lceil\frac{p(k)}8\right\rceil\cdot\left\lceil\frac m8
\right\rceil
\end{equation}
for the number of squares in $A$, which is roughly $\frac18$ of the
bound mentioned in \ref{improvbound}. If $\min\{p(k),\,m\}$ is at
most~$4$ (resp.~$2$), we get a better valid bound by replacing the
number~$8$ in \eqref{genbd} by~$4$ (resp.~$2$). By making use of the
fact that $c_1=1$, all these bounds can still be improved a little
more, similar as in the previous remark.
\end{rems}

\begin{lab}
The qualitative part of the above result extends immediately to the
following more general situation. For any commutative ring $B$, let
$\Sigma B^2$ denote the set of sums of squares in $B$. Let $K/k$ be a
field extension and let $A$ be a (commutative) $k$-algebra. Fix
elements $h_1,\dots,h_r\in A$, amd consider the so-called (pseudo)
quadratic module
$$M\>:=\>\Bigl\{\sum_{i=1}^rs_ih_i\colon s_1,\dots,s_r\in\Sigma A^2
\Bigr\}$$
generated in $A$ by the $h_i$. Similarly, let
$$M_K\>:=\>\Bigl\{\sum_{i=1}^rt_ih_i\colon t_1,\dots,t_r\in\Sigma A_K
^2\Bigr\}$$
be the (pseudo) quadratic module generated by $M$ in $A_K=A\otimes_k
K$. Then we have:
\end{lab}

\begin{prop}\Label{mkvsm}%
In the above situation, if $K/k$ is a finite extension of real fields
such that every ordering of $k$ extends to $[K:k]$ different
orderings of $K$, we have $A\cap M_K=M$.
\end{prop}

\begin{proof}
Let $t_1,\dots,t_r\in\Sigma A_K^2$ be such that $f:=\sum_{i=1}^rt_i
h_i$ lies in $A$. Taking the trace of $f$ gives
$$f\>=\>\frac1d\sum_{i=1}^r\tr_{A_K/A}(t_i)\,h_i.$$
For any $i=1,\dots,r$, the trace $\tr_{A_K/A}(t_i)$ lies in $\Sigma
A^2$, see \ref{tracerapp}. It follows that $f\in M$.
\end{proof}


\section{Construction of counterexamples}\Label{sect:stq}%

We construct a family of forms with rational coefficients which are
sums of squares over $\R$ but not over~$\Q$:

\begin{thm}\Label{stq}%
Let $n\ge2$, and let $d\ge4$ be an even number. There exists a form
$f\in\Q[x_0,\dots,x_n]$ of degree~$d$ with the following properties:
\begin{itemize}
\item[(1)]
$f$ is irreducible over $\Q$, and decomposes into a product of $d$
linear forms over $\C$;
\item[(2)]
$f$ is a sum of two squares in $\R[x_0,\dots,x_n]$;
\item[(3)]
$f$ is not a sum of any number of squares in $\Q[x_0,\dots,x_n]$.
\end{itemize}
\end{thm}

For example,
$$f\>=\>x_0^4+x_0x_1^3+x_1^4-3x_0^2x_1x_2-4x_0x_1^2x_2+2x_0^2x_2^2+
x_0x_2^3+x_1x_2^3+x_2^4$$
is such a form.

\begin{lab}\Label{gensetup}%
To prove the theorem we consider the following setup.
Let $K$ be a totally imaginary number field of degree~$d=2m$,
let $E$ be the Galois hull of $K/\Q$, and let $G=\Gal(E/\Q)$ (resp.\
$H=\Gal(E/K)$) be the Galois group of $E$ over~$\Q$ (resp.\ of $E$
over~$K$).
The group $G$ acts transitively on the set $\Hom(K,E)$ of embeddings
$K\to E$ by
$$({}^\sigma\varphi)(\alpha)\>=\>\sigma(\varphi(\alpha))\quad(\alpha
\in K)$$
($\sigma\in G$, $\varphi\in\Hom(K,E)$), thereby identifying the
$G$-set $\Hom(K,E)$ with $G/H$.
Note that $|G/H|=d$.
We fix an embedding $E\into\C$ and denote by $\tau\in G$ the
restriction of complex conjugation to~$E$. Since $K$ is totally
imaginary, $\tau$ acts on $G/H$ without fixpoint.
\end{lab}

\begin{lab}\Label{linformgen}%
We extend the $G$-action on $E$ to an action on $E[\x]=E[x_0,\dots,
x_n]$ by letting $G$ act on the coefficients. Let $l\in K[\x]$ be a
linear form, and let $L\subset\P^n$ be the hyperplane $l=0$. We
assume that the $d$ Galois conjugates of $L$ are in general position,
that is, the intersection of any $r\le n+1$ of them has
codimension~$r$ (the empty set is assigned the codimension~$n+1$).
For example, this condition is satisfied when $\alpha$ is a primitive
element for $K/\Q$ and
$$l\>=\>\sum_{i=0}^n\alpha^ix_i,$$
as one sees by a Vandermonde argument.
We consider the form
\begin{equation}\Label{dfnf}%
f\>:=\>\prod_{\sigma H\in G/H}{}^\sigma l\>=\>N_{K/\Q}(l)
\end{equation}
of degree $d$. Clearly, $f$ has rational coefficients and is
irreducible over $\Q$.
Moreover, since $\tau$ acts on $G/H$ without fixpoint, we can choose
$m=\frac d2$ cosets $\sigma_1H,\dots,\sigma_mH$ in $G/H$ which
represent the $\tau$-orbits. Writing $l_j:={}^{\sigma_j}l$
($j=1,\dots,m$)
we therefore have
$$f\>=\>\prod_{i=1}^ml_i\ol{l_i}$$
where bar denotes coefficientwise complex conjugation. This shows
that $f$ is a product of $m$ quadratic forms over $\R$, each of which
is a sum of two squares over $\R$. In particular, $f$ is a sum of two
squares in $\R[\x]$.
\end{lab}

\begin{lab}\Label{keyarg}%
Let us label the $d=2m$ hyperplanes ${}^\sigma l=0$ ($\sigma H\in
G/H$) by $L_1,\dots,L_d$. By our assumption of general position, the
$\choose d2$ pairwise intersections $M_{ij}=L_i\cap L_j$ ($1\le i<j
\le d$) are all distinct, and are linear subspaces of $\P^n$ of
codimension two.
Exactly $m$ of the $M_{ij}$ are conjugation-invariant, and they
correspond to the $\tau$-orbits in $G/H$.
We say that $M_{ij}$ is real if it is conjugation-invariant.

We now assume that the action of $G$ on $G/H$ is $2$-transitive.
Then $G$ acts transitively on the set $\{M_{ij}\colon1\le i<j\le
d\}$. We claim that $f$ cannot be a sum of squares of forms with
rational coefficients.
To see this, suppose
$$f\>=\>p_1^2+\cdots+p_r^2$$
where $p_1,\dots,p_r$ are forms of degree $m$ in $\Q[\x]$. Each
$p_\nu$ vanishes identically on the $m$ real intersections $M_{ij}$.
By Galois invariance and by the transitivity assumption, the $p_\nu$
have to vanish identically on all $\choose d2$ intersections
$M_{ij}$. But there is no nonzero form of degree $m$ with this
property. In fact, we have $m\le d-2$ since $d=2m\ge4$, and there is
not even any nonzero such form of degree $d-2$. This follows from the
next lemma, which we state in a stronger version with a view to a
later application:
\end{lab}

\begin{lem}\Label{intptslines2}%
Let $k$ be a field and $\x=(x_0,\dots,x_n)$ with $n\ge2$, and let
$l_1,\dots,l_d\in k[\x]$ be linear forms such that the hyperplanes
$L_i=\{l_i=0\}$ ($i=1,\dots,d$) are in general position. Let $I$ be
the vanishing ideal of $\bigcup_{1\le i<j\le d}L_i\cap L_j$ in
$k[\x]$. Then $I$ is generated by the $d$ forms
$$p_i\>:=\>\frac{l_1\cdots l_d}{l_i}\quad(i=1,\dots,d)$$
of degree $d-1$.
\end{lem}

In particular, for $d\ge3$ there is no hypersurface of degree $d-2$
containing $L_i\cap L_j$ for all $1\le i<j\le d$.

\begin{proof}
The assertion is obviously true for $d\le2$, so we can assume that
$d\ge3$ and the lemma is proved for smaller values of~$d$. Clearly we
have $(p_1,\dots,p_d)\subset I$.
Conversely let $g\in I$ be a form. Since $L_1\cap L_2,\,\dots,\,
L_1\cap L_d$ are distinct hypersurfaces in $L_1$, and since $g$
vanishes on all of them, we see that $g$ is a multiple of $l_2\cdots
l_d=p_1$ modulo~$l_1$, that is, $g=g_1p_1+l_1h$ with suitable forms
$g_1$ and $h$. The form $h$ vanishes on the pairwise intersections of
the hypersurfaces $L_2,\dots,L_d$.
Writing $q_i:=(l_2\cdots l_d)/l_i$ for $i=2,\dots,d$, it follows from
the inductive hypothesis that $h\in(q_2,\dots,q_d)$. Since $l_1q_i=
p_i$ for $i=2,\dots,d$, we conclude $g\in(p_1,\dots,p_d)$, as
desired.
\end{proof}

Let $\ol\Q$ denote an algebraic closure of $\Q$. Summarizing, we have
proved:

\begin{thm}\Label{thmprec}%
Let $n\ge2$, let $K/\Q$ be a totally imaginary number field of degree
$d\ge4$, and let $l\in K[x_0,\dots,x_n]$ be a linear form whose $d$
Galois conjugates over $\Q$ are in general position. If the action of
$\Gal(\ol\Q/\Q)$ on $\Hom(K,\C)$ is $2$-transitive, then
$$f\>:=\>N_{K/\Q}(l)$$
is a form of degree~$d$ with rational coefficients that is
irreducible over $\Q$ and a sum of two squares over $\R$, but not a
sum of any number of squares over~$\Q$.
\qed
\end{thm}

\begin{lab}
Clearly, this implies the statement of Theorem \ref{stq}: We may
start with any totally imaginary number field $K$ of degree $d\ge4$
for which the Galois action on $\Hom(K,\C)$ is $2$-transitive. For
example, the Galois group may act as the alternating or full
symmetric group on $d$ letters. Picking any primitive element
$\alpha$ of $K/\Q$, the form $f$ constructed as in \ref{linformgen}
satisfies all the properties of \ref{stq}.
\end{lab}

\begin{example}\Label{bspprec}%
To produce an explicit example, take the field $K=\Q(\alpha)$ where
$\alpha^4-\alpha+1=0$. In this case the Galois group acts as the full
symmetric group on the roots of $t^4-t+1$, as one sees by reducing
modulo~$2$ and modulo~$3$.
Starting with $l=x_0+\alpha x_1+\alpha^2x_2$, one obtains the form
$f$ displayed after Theorem \ref{stq}.

To see a sum of squares representation of $f$ explicitly, let $\beta$
be a root of $t^3-4t-1=0$ (the cubic resolvent of $t^4-t+1$). Then
the following decomposition holds:
$$4f\>=\>\Bigl(2x_0^2+\beta x_1^2-x_1x_2+(2+\frac1\beta)x_2^2\Bigr)^2
-\beta\Bigl(2x_0x_1-\frac{x_1^2}\beta+\frac{2x_0x_2}\beta+\beta
x_1x_2-x_2^2\Bigr)^2.$$
The cubic field $\Q(\beta)$ is totally real, but not Galois over
$\Q$.
Its three places send $\beta$ to real numbers approximately equal to
$$-1.860805854,\quad-0.2541016885,\quad2.114907542,$$
respectively. Therefore, the first two embeddings give each a
representation of $f$ as a sum of two squares of real quadratic
forms. These representations are defined over the real field $F=
\Q(\sqrt{-\beta})$ of degree six. Up to equivalence, these are the
only two representations of $f$ over $\R$ as a sum of two squares.
Every other sum of squares representation of $f$ over $\R$ is
(equivalent to) a sum of four squares, and arises as a convex
combination of the two extremal representations.
\end{example}

\begin{rem}\Label{2transweakened}%
For the conclusion of Theorem \ref{thmprec}, it is not necessary that
$G$ acts $2$-transitively on $G/H$, or equivalently, that $G$ acts
transitively on the set $\{M_{ij}\colon1\le i<j\le d\}$ (see
\ref{keyarg}). It suffices that any $G$-orbit in this set contains at
least one space $M_{ij}$ which is real, i.e., invariant under complex
conjugation $\tau$. In terms of the $G$-action on $G/H$, this means
the following condition:
\begin{quote}
\begin{itemize}
\item[$(*)$]
\emph{For any $x,\,y\in G/H$ with $x\ne y$ there exist $z\in G/H$ and
$\sigma\in G$ such that $x=\sigma z$ and $y=\sigma\tau z$.}
\end{itemize}
\end{quote}
For $d=|G/H|=4$, condition $(*)$ implies $2$-transitivity of $G$ on
$G/H$. But for $d\ge6$ there are examples where $G$ satisfies $(*)$
without being $2$-transitive. The simplest such example is given by
the group $G$ of rotations of a regular cube $P$, acting on the set
$F$ of (two-dimensional) faces. So $G=S_4$, the symmetric group on
four letters, and $H$ is the cyclic subgroup generated by a $4$-cycle
in $G$. A pair $\{f,f'\}$ of different faces of $P$ consists either
of two faces with a common edge, or of two opposite faces. Hence
there are exactly two $G$-orbits in the set $\choose F2$ of pairs of
faces.
The involution $\tau$ is the rotation of order two around an axis
that joins the midpoints of two opposite edges. Among the three pairs
$\{f,\tau f\}$ ($f\in F$) of faces, one consists of opposite faces,
while the other two consist of adjacent faces. So each pair of faces
is $G$-conjugate to a pair of the form $\{f,\tau f\}$.

An example of a (totally imaginary) number field which realizes this
Galois action on its set of places is $k=\Q(\alpha)$ with
$$\alpha^6-\alpha^5+2\alpha^4+\alpha^3+2\alpha^2+3\alpha+1\>=\>0.$$
The example was found by consulting the Bordeaux number field tables
\cite{bnf}.
\end{rem}

\begin{rem}\Label{generalnbfd}%
We can easily extend Theorem \ref{stq} to real number fields other
than $\Q$. Indeed, let $K$, $E$, $l$, $f$ etc.\ be as in
\ref{linformgen}, and assume that $G=\Gal(E/\Q)$ acts
$2$-transitively on $\Hom(K,E)$. Let $k$ be any number field with at
least one real place, and consider the natural embedding $\phi$ from
$\Gal(kE/k)$ into $G$, induced by restriction of automorphisms.
Then $\phi$ is surjective if (and only if) $E\cap k=\Q$, that is, if
$E$ and $k$ are linearly disjoint over $\Q$. Assuming that this is
the case, we claim that $f$ is not a sum of squares over $k$. Indeed,
by the argument in \ref{keyarg}, the $\choose d2$ intersections
$M_{ij}=L_i\cap L_j$ are all Galois conjugate among each other over
$k$. If there were an identity $f=p_1^2+\cdots+p_r^2$ with forms
$p_\nu\in k[\x]$, the $p_\nu$ would have to vanish on the union of
the $M_{ij}$, which again is impossible by Proposition
\ref{intptslines2}.

Using this way of reasoning, we conclude:
\end{rem}

\begin{cor}\Label{stqgend}%
Let $k$ be any fixed number field with at least one real place, let
$n\ge2$ and $d\ge4$ be even. Then there exists a form $f\in\Q[x_0,
\dots,x_n]$ of degree~$d$ which is a sum of two squares of forms over
$\R$, but not a sum of squares of forms over~$k$.
\end{cor}

In particular, over a real number field there is no analogue of
Hilbert's theorem \cite{h} over $\R$, according to which every
nonnegative ternary quartic form is a sum of squares of quadratic
forms.

\begin{proof}
Given $k$, it suffices by \ref{generalnbfd} to find a totally
imaginary extension $K/\Q$ with Galois hull $E/\Q$ for which $G=\Gal
(E/\Q)$ acts $2$-transitively on $\Hom(K,E)$, and such that $E$ and
$k$ are linearly disjoint. The latter will certainly be the case if
the discriminants of $E$ and $k$ are relatively prime. So the
assertion follows from the next lemma.
\end{proof}

\begin{lem}
For any finite set $S$ of primes and any even number $d$, there
exists a totally imaginary number field $K/\Q$ of degree~$d$ with
Galois hull $E/\Q$, such that the discriminant of $E$ is not
divisible by any prime in $S$, and such that $\Gal(E/\Q)$ is
$2$-transitive on $\Hom(K,E)$.
\end{lem}

\begin{proof}
It suffices to find a monic polynomial $g(x)$ over $\Z$ of degree~$d$
with the following properties: (1)~$g$~is positive definite; (2)~the
discriminant of $g$ is not divisible by any prime in~$S$; (3)~there
exist primes $p$, $q$ such that $g$ mod~$p$ is irreducible and $g$
mod~$q$ is a linear factor times an irreducible polynomial. Given
such $g$, let $K=\Q(\alpha)$ where $\alpha$ is a root of $g$, and let
$E$ be the Galois hull of $K$. Then $K$ has the required properties.
In particular, the action of $G=\Gal(E/\Q)$ on the roots of $g$ is
$2$-transitive since $G$ contains a $(d-1)$-cycle.
Properties (2) and (3) can be guaranteed by arranging a particular
factor decomposition of $g$ modulo~$p$, for finitely many primes~$p$.
So it is clear that (many) polynomials $g$ as above can be found.
\end{proof}


\section{Rational denominators}\Label{sect:denoms}%

\begin{lab}
Let $\x=(x_0,\dots,x_n)$ with $n\ge2$, and let $f\in\Q[\x]$ be a form
of degree $d$ as constructed in Theorem \ref{stq}. In particular, $f$
is a sum of squares over $\R$, but not over $\Q$. By Artin's solution
\cite{a} to Hilbert's 17th problem, $f$ is a sum of squares of
rational functions over $\Q$. In other words, there exists a form
$h\ne0$ in $\Q[\x]$ such that both $h$ and $fh$ are sums of squares
over $\Q$. When $f$ is constructed explicitly as in Section
\ref{sect:stq}, what can be said about the degree of such~$h$? Is it
possible to give explicit constructions for~$h$?

These questions were raised by M.-F.~Roy. We will give a complete
answer for $d=4$, and a partial answer for $d\ge6$, see Theorem
\ref{denomh17} and Proposition \ref{beautiful}.
\end{lab}

\begin{lab}\Label{setup}%
For the following we assume the setup of Theorem \ref{thmprec}. Hence
$K/\Q$ will be a totally imaginary number field of degree $d\ge4$,
with Galois hull $E/\Q$. Letting $G=\Gal(E/\Q)$ and $H=\Gal(E/H)$, we
assume that the action of $G$ on $G/H$ is $2$-transitive. (In fact,
it will suffice for the following to have the weaker condition
\ref{2transweakened} satisfied.) Let $l\in K[\x]$ be a linear form
such that the $d$ Galois conjugates of the hyperplane $l=0$ are in
general position. The form
$$f\>=\>\prod_{\sigma H\in G/H}{}^\sigma l\>=\>N_{K/\Q}(l)$$
in $\Q[\x]$ satisfies the conclusions of Theorem \ref{thmprec}.
\end{lab}

\begin{thm}\Label{denomh17}%
Let $f$ be as in \ref{setup}.
\begin{itemize}
\item[(a)]
There exists a nonzero form $h\in\Q[\x]$ of degree $d-2$, but not of
smaller degree, for which $fh$ is a sum of squares of forms in
$\Q[\x]$.
\item[(b)]
When $d=4$, any $h$ as in (a) is a sum of squares of linear forms in
$\Q[\x]$.
\end{itemize}
\end{thm}

\begin{proof}
Let $0\ne h\in\Q[\x]$ be a form for which $fh$ is a sum of squares of
forms over $\Q$, say $fh=g_1^2+\cdots+g_r^2$ with forms $g_1,\dots,
g_r\in\Q[\x]$. Let $l=l_1,l_2,\dots,l_d\in\ol\Q[\x]$ be the linear
forms that are Galois conjugate to $l$, let $L_i$ be the hyperplane
$l_i=0$, and consider the pairwise intersections $L_i\cap L_j$ ($1\le
i<j\le d$). The forms $g_\nu$ vanish identically on those
intersections $L_i\cap L_j$ that are real (there are $\frac d2$
such). Since the action of $G$ permutes the $L_i\cap L_j$
transitively, and since the $g_\nu$ have rational coefficients, the
$g_\nu$ vanish identically on $\bigcup_{i<j}L_i\cap L_j$. The
vanishing ideal $I$ of this union (inside $\ol\Q[\x]$) is generated
by the forms $p_i:=\frac f{l_i}$ ($i=1,\dots,d$), by Lemma
\ref{intptslines2}. We already conclude that $\deg(g_\nu)\ge d-1$,
and hence $\deg(h)\ge d-2$.

It remains to construct a form $h$ of exact degree $d-2$ for which
$fh$ is a sum of squares over $\Q$. Note that assertion (b) is clear
from (a), since here $h$ is a quadratic form over $\Q$ and is
nonnegative on $\R^{n+1}$.

The proof of Theorem \ref{denomh17} will therefore be completed by
the next proposition. It gives a fully explicit rendering of the
theorem:
\end{proof}

\begin{prop}\Label{beautiful}%
Let $f$ be as in \ref{setup}. For a form $g\in\Q[\x]$ of degree
$2d-2$, the following conditions are equivalent:
\begin{itemize}
\item[(i)]
$g$ is divisible by $f$ and is a sum of squares of forms in $\Q[\x]$;
\item[(ii)]
there exist $r\ge1$ and elements $a_1,\dots,a_r\in K$ with $a_1^2+
\cdots+a_r^2=0$ such that
$$g\>=\>\sum_{\nu=1}^r\Bigl(\tr_{K/\Q}\bigl(\frac{a_\nu f}l\bigr)
\Bigr)^2.$$
\end{itemize}
If (i) and (ii) hold, then conversely every sum of squares
representation of $g$ in $\Q[\x]$ has the form stated in (ii), for
suitable $a_\nu\in K$ with $\sum_\nu a_\nu^2=0$.
\end{prop}

\begin{proof}
Assume $g=g_1^2+\cdots+g_r^2$ where $g_1,\dots,g_r\in\Q[\x]$ are
forms of degree $d-1$, and assume that $g$ is divisible by $f$. As
before, let $I\subset\ol\Q[\x]$ be the vanishing ideal of $\bigcup
_{i<j}L_i\cap L_j$. By Lemma \ref{intptslines2} we have $g_1,\dots,
g_r\in I\cap\Q[\x]$.

Let $\tr=\tr_{K/\Q}$ be the trace of $K$ over $\Q$. We claim that a
form $g\in I$ of degree $d-1$ has $\Q$-coefficients if and only if
$g=\tr(af/l)$ for some $a\in K$. Indeed, let
$$g\>=\>b_1\frac f{l_1}+\cdots+b_d\frac f{l_d}$$
with $b_i\in\ol\Q$. Let us label the elements of $G/H$ as $\sigma_1H,
\dots,\sigma_dH$ in such a way that $\sigma_1=1$ and $l_i=
{}^{\sigma_i}l$ for $i=1,\dots,d$. The forms $\frac f{l_1},\dots,
\frac f{l_d}$ are linearly independent.
We conclude that $g$ lies in $\Q[\x]$ if and only if $b_1\in K$ and
$b_i=\sigma_i(b_1)$ for $i=1,\dots,d$, or in other words, if and only
if $g=\tr(b_1f/l)$ with $b_1\in K$.

It remains to characterize when a sum of squares
$$g\>=\>\sum_{\nu=1}^r\Bigl(\tr\bigl(\frac{a_\nu f}l\bigr)\Bigr)^2$$
with $a_1,\dots,a_r\in K$ is divisible by $f=l_1\cdots l_d$, or
equivalently, by~$l=l_1$. Since
$$g\>=\>\sum_{\nu=1}^r\Bigl(\sigma_1(a_\nu)\frac f{l_1}+\cdots+
\sigma_d(a_\nu)\frac f{l_d}\Bigr)^2,$$
we see that $g$ is divisible by $l_1$ if and only if $\sum_\nu
a_\nu^2=0$.

The proof of Proposition \ref{beautiful}, and therefore of Theorem
\ref{denomh17}, is complete.
\end{proof}

\begin{rem}
In (a) of Theorem \ref{denomh17}, we can always find $0\ne h\in
\Q[\x]$ of degree $d-2$ such that $fh$ is a sum of five squares in
$\Q[\x]$. This follows from \ref{beautiful} since $-1$ is a sum of
four squares in $K$. If $K$ happens to have level~$2$, i.e., if $-1$
is a sum of two squares in $K$, then $fh$ can be made a sum of three
squares in $\Q[\x]$. (Note that $-1$ cannot be a square in $K$, and so
$fh$ cannot be made a sum of two squares in $\Q[\x]$.)
\end{rem}

\begin{example}
To illustrate the preceding construction, let us review the example
of a ternary quartic $f\in\Q[x_0,x_1,x_2]$ given after Theorem
\ref{stq} (c.f.\ also Example \ref{bspprec}). In this case, the
number field $K=\Q(\alpha)$ with $\alpha^4-\alpha+1=0$ has level~$2$,
as one can conclude from general reasons since the prime~$2$ is inert
in~$K$.
Explicitly, this is confirmed by the identity
$$(\alpha^2+\alpha-1)^2+(\alpha^2-\alpha)^2+1\>=\>0$$
in $K$. Writing
$$g_1\>=\>\tr\Bigl(\frac fl\Bigr),\quad g_2\>=\>
\tr\Bigl(\frac{(\alpha^2+\alpha-1)f}l\Bigr),\quad g_3\>=\>
\tr\Bigl(\frac{(\alpha^2-\alpha)f}l\Bigr)$$
with $\beta=\alpha^2+\alpha-1$ and $\gamma=\alpha^2-\alpha$, we find
\begin{fitbigdiagram}
\begin{align*}
g_1= & \ 4x_0^3+x_1^3+x_2^3+4x_0x_2^2-4x_1^2x_2-6x_0x_1x_2, \\
g_2= & \ -4x_0^3+3x_1^3+4x_2^3-3x_0^2x_1-x_0x_1^2+x_0^2x_2-x_0x_2^2+
  4x_1^2x_2+3x_1x_2^2-2x_0x_1x_2, \\
g_3= & \ -4x_1^3+3x_2^3-3x_0^2x_1-7x_0x_1^2+7x_0^2x_2+3x_0x_2^2+
  3x_1x_2^2+8x_0x_1x_2.
\end{align*}
\end{fitbigdiagram}%
Expanding the sum of squares, we obtain
$$fh\>=\>g_1^2+g_2^2+g_3^2$$
where
$$h\>=\>32\,x_0^2+24\,x_0x_1-8\,x_0x_2+26\,x_1^2+16\,x_1x_2+26\,x_2^2.$$
To write $h$ as a sum of squares in an explicit way, we may observe
$$86\,h\>=\>43\,(8x_0+3x_1-x_2)^2+(43x_1+19x_2)^2+1832\,x_2^2.$$
\end{example}


\section{Ternary quartics}\Label{sect:tq}

In this section we restrict to ternary forms of degree four. It was
proved by Hilbert in 1888 \cite{h} that every nonnegative quartic
form $f\in\R[x_0,x_1,x_2]$ is a sum of squares of quadratic forms
(and, in fact, of three squares). In Theorem \ref{stq} we constructed
a family of quartic forms $f\in\Q[x_0,x_1,x_2]$ that are sums of
squares over $\R$, but not over $\Q$. Now we'll show that conversely
every nonnegative ternary quartic $f\in\Q[x_0,x_1,x_2]$ that fails to
be a sum of squares over $\Q$ arises from the construction in
\ref{stq}. More precisely, we'll prove:

\begin{thm}\Label{terquaronly}%
Let $f\in\Q[x_0,x_1,x_2]$ be a nonnegative form of degree~$4$ which
is not a sum of squares over $\Q$. Then $f$ is a product $f=l_1l_2
l_3l_4$ of linear forms in $\C[x_0,x_1,x_2]$, the four lines $l_i=0$
are in general position, and $\Gal(\ol\Q/\Q)$ acts on the set of
these lines
as the symmetric or alternating group on four letters.
\end{thm}

Before starting the proof, we need to introduce an important general
concept. In the sequel let $\x=(x_0,\dots,x_n)$ with arbitrary
$n\ge1$, and denote by $\Sigma$ the cone of sums of squares in
$\R[\x]$.

\begin{dfn}\Label{dfncharsubsp}%
Given a sum of squares $f\in\Sigma$, the set
$$U_f\>:=\>\{p\in\R[\x]\colon f-\epsilon p^2\in\Sigma\text{ for
some }\epsilon>0\}$$
will be called the \emph{characteristic subspace} for $f$.
\end{dfn}

\begin{lem}\Label{ufprops}%
Let $f\in\Sigma$.
\begin{itemize}
\item[(a)]
The set $\{p\in\R[\x]\colon f-p^2\in\Sigma\}$ is convex. Hence $U_f$
is a linear subspace of $\R[\x]$.
\item[(b)]
There is a sum of squares representation $f=p_1^2+\cdots+p_r^2$ of
$f$ in which $p_1,\dots,p_r$ is a linear basis of $U_f$.
\end{itemize}
\end{lem}

\begin{proof}
If $f-p_j^2\in\Sigma$ for $j=1,2$, then
$$f-((1-t)p_1+tp_2)^2\>=\>(1-t)(f-p_1^2)+t(f-p_2^2)+t(1-t)(p_1-p_2)
^2\>\in\>\Sigma$$
for $0\le t\le1$, proving (a).
As for (b), there is a basis $q_1,\dots,q_r$ of $U_f$ such that
$f-q_i^2\in\Sigma$ for $i=1,\dots,r$. By averaging over corresponding
sums of squares expressions we find a sum of squares representation
$f=g_1^2+\cdots+g_k^2$ in which $g_1,\dots,g_k$ span $U_f$. Now
diagonalizing the symmetric tensor $\sum_{i=1}^kg_i\otimes g_i\in U_f
\otimes U_f$ gives the assertion.
\end{proof}

The reason why the characteristic subspaces will be useful here is
the following lemma (which generalizes \cite{hi} Theorem 1.2):

\begin{lem}\Label{ufdefqsosq}%
Let $f\in\Q[\x]\cap\Sigma$, i.e., $f$ is a polynomial with rational
coefficients and is a sum of squares in $\R[\x]$. If the subspace
$U_f$ of $\R[\x]$ is defined over $\Q$, then $f$ is a sum of squares
in $\Q[\x]$.
\end{lem}

(If $V$ is a $\Q$-vector space, a linear subspace $L$ of $V\otimes_\Q
\R$ is said to be defined over $\Q$ if it is spanned by $L\cap V$.
Similarly for affine-linear subspaces.)

\begin{proof}
Let $U_f\subset\R[\x]$ be the characteristic subspace of $f$, let
$S^2U_f$ be its second symmetric power, and let $\gamma\colon S^2U_f
\to\R[\x]$ be the natural linear (product) map. Since $U_f$ is
defined over $\Q$, so is $\Gamma_f:=\gamma^{-1}(f)$, an affine-linear
subspace of $S^2U_f$. By Lemma \ref{ufprops}(b), $\Gamma_f$ contains
an element of $S^2U_f$ that is positive definite. From density of
$\Q$ in $\R$ we conclude that $\Gamma_f$ also contains a positive
definite element defined over $\Q$. In particular, hence, $f$ is a
sum of squares over $\Q$.
\end{proof}

Hilbert's theorem \cite{h} on ternary quartics allows us to give an
easy geometric descriptions for the characteristic subspaces of
ternary quartics. First, the problem is local:

\begin{lem}\Label{lgpcharsubsp}%
Let $f,\,g\in\R[\x]$ be two nonnegative forms of the same degree.
Assume for every $0\ne\xi\in\R^{n+1}$ that there exists $\epsilon>0$
for which $f-\epsilon g$ is nonnegative in a neighborhood of $\xi$.
Then there exists $\epsilon>0$ such that $f-\epsilon g$ is
nonnegative on $\R^{n+1}$.
\end{lem}

\begin{proof}
This follows from compactness of projective space: For any $\xi\in
\P^n(\R)$ there exists $\epsilon_\xi>0$ and a neighborhood $W_\xi
\subset\P^n(\R)$ of $\xi$ such that $f-\epsilon_\xi g$ is nonnegative
on $W_\xi$. Choose finitely many points $\xi_1,\dots,\xi_r\in
\P^n(\R)$ such that $\P^n(\R)=\bigcup_{i=1}^rW_{\xi_i}$, and put
$\epsilon=\min\{\epsilon_{\xi_i}\colon i=1,\dots,r\}$. Then
$f-\epsilon g$ is everywhere nonnegative.
\end{proof}

From now on let $\x=(x_0,x_1,x_2)$.
For a nonnegative ternary quartic $f\in\R[\x]$ with isolated real
zeros, we will determine the characteristic subspace $U_f$
explicitly. (The case where the real zeros are not isolated is even
easier, since it reduces to nonnegative quadratic forms.) By Lemma
\ref{lgpcharsubsp}, it suffices to do this locally, namely to
determine the subspace
$$U_{f,\xi}\>:=\>\{p\in\R[\x]_2\colon\ex\epsilon>0\ f-\epsilon p^2
\ge0\text{ around }\xi\}$$
for every $\xi\in\P^2(\R)$ with $f(\xi)=0$. Note that $U_{f,\xi}$ is
also the space of all $p\in\R[\x]_2$ for which $p^2/f$ is locally
bounded around~$\xi$ in $\P^2(\R)$. (Here $\R[\x]_2$ denotes the
space of quadratic forms in $\R[\x]$.)

Assume that $\xi$ is an isolated real zero of a quartic form
$f\in\R[\x]$. Then $\xi$ is a singularity of the curve $f=0$ of real
type $A_1^*$, $A_3^*$, $A_5^*$, $A_7^*$ or $X_9^{**}$. The last two
can occur only when $f$ is reducible over $\C$. Here, by a (plane)
$A_k^*$-singularity (for $k\ge1$ odd), we mean a real analytic
singularity of type $A_k$ whose two analytic branches are complex
conjugate.
See \ref{remx9} for $X_9^{**}$.

\begin{prop}\Label{aklocbd}%
Let $f(x,y)$, $p(x,y)$ be real analytic function germs in $(\R^2,0)$,
and assume that the singularity $f=0$ is of type $A_{2r-1}^*$ with
$r\ge1$. Then the germ $p^2/f$ is locally bounded in $\R^2$ around
$0$ if and only if $i(f,p)\ge2r$, where $i$ denotes the local
intersection number at $0\in\R^2$.
\end{prop}

\begin{proof}
We may assume $f=y^2+x^{2r}$, and we'll show that both properties
are equivalent to $\omega(p(x,0))\ge r$, where $\omega$ is the
vanishing order at $x=0$. It is clear that $\frac{p^2}f$ locally
bounded implies $\omega(p(x,0))\ge r$.
Conversely, if $\omega(p(x,0))\ge r$, we can write $p=yg+x^rh$ with
analytic germs $g,\,h$. Then a simple calculation shows that $\frac
{p^2}f$ is locally bounded around $0\in\R^2$.
On the other hand, from $f=(y+ix^r)(y-ix^r)$ we can directly deduce
that $\omega(p(x,0))\ge r$ if and only if $i(f,p)\ge2r$.
\end{proof}

\begin{rem}\Label{remx9}%
A plane real singularity of type $X_9^{**}$ corresponds to the
union of four nonreal lines through a real point (two pairs of
complex conjugate lines), see \cite{agzv} p.~185. A normal form is
given by $f=x^4+y^4+ax^2y^2$ with $a>0$, $a\ne2$. For a real analytic
germ $p(x,y)$, the quotient $\frac{p^2}f$ is locally bounded iff
$\omega(p(x,y))\ge2$ iff $i(f,p)\ge8$.
\end{rem}

\begin{cor}\Label{bigcapufxidefq}%
Let $f\in\Q[\x]$ be a nonnegative ternary quartic over $\Q$, let
$\xi_1,\dots,\xi_r$ be isolated real zeros of $f$ in $\P^2$, and
assume that the set $\{\xi_1,\dots,\xi_r\}$ is invariant under the
action of $\Gal(\ol\Q/\Q)$ on $\P^2(\ol\Q)$. Then the subspace
$\bigcap_{j=1}^rU_{f,\xi_j}$ of $\R[\x]_2$ is defined over $\Q$.
\end{cor}

Note that $\xi_1,\dots,\xi_r$ have coordinates in $\ol\Q$, so the
Galois group acts on these points.

\begin{proof}
By Hilbert's theorem \cite{h}, this is clear from \ref{aklocbd} and
\ref{remx9}.
\end{proof}

\begin{lab}
We now give the proof of Theorem \ref{terquaronly}. Let $\x=(x_0,
x_1,x_2)$, let $f\in\Q[\x]$ be a nonnegative form of degree~$4$. Let
$U_f\subset\R[\x]_2$ be the characteristic subspace of $f$ (see
\ref{dfncharsubsp}). By a case distinction we will show that $f$ is a
sum of squares over $\Q$ unless it satisfies the conditions of
Theorem \ref{terquaronly}.

Whenever $U_f$ is defined over $\Q$, $f$ is a sum of squares over
$\Q$ by Lemma \ref{ufdefqsosq}. In particular, this is the case when
$f$ is strictly positive definite, since then $U_f=\R[\x]_2$. So we
assume that $f$ has at least one real zero. We first consider the
case where $f$ is absolutely irreducible. The real zeros of $f$ are
precisely the real singular points of the curve $f=0$. The
configuration of all (real or nonreal) singularities of this curve is
one of the following (see \cite{sch:jag} 7.3):
$$A_1^*,\ 2A_1^*,\ 3A_1^*,\ A_3^*,\ A_1^*+A_3^*,\ A_5^*,\
A_1^*+2A_1^i,\ A_1^*+2A_2^i.$$
(Here $2A^i_k$ denotes a pair $P\ne \ol P$ of complex conjugate
$A_k$-singularities.) The singularities are permuted by the Galois
action. In all cases except the last two, every singularity of $f$ is
real. By Lemma \ref{lgpcharsubsp} and Corollary \ref{bigcapufxidefq},
the subspace $U_f$ is defined over $\Q$ in these cases, and we are
done.
In the case of $A_1^*+2A_2^i$, the same is true since the unique real
singularity is Galois invariant, hence defined over~$\Q$.

It remains to consider the case where $f$ has three nodes, one of
which is real (with a pair of nonreal tangents) and the other two are
complex conjugate. Here $U_f$ consists of the quadratic forms with a
zero in the real node, and we see that $U_f$ fails to be defined over
$\Q$.
Instead we can argue as follows: For such $f$, there exists a unique
(up to orthogonal equivalence) representation $f=p_1^2+p_2^2+p_3^2$
in $\C[\x]$ for which $p_1,p_2,p_3$ vanish in all three nodes.
Moreover, the symmetric tensor $t:=\sum_{j=1}^3p_j\otimes p_j$ is
defined over $\R$ and is positive semidefinite. This follows from the
analysis in \cite{sch:jag} (for more details see \cite{sch:tqtables},
pp.~4 and~6).
Since the set of all three nodes is Galois invariant, the tensor $t$
is defined over $\Q$, and hence $f$ is a sum of squares over $\Q$. We
have thus shown that $f$ is a sum of squares over $\Q$ when $f$ is
absolutely irreducible.
\end{lab}

\begin{lab}\Label{reduc1}%
It is easily seen that $f$ is a sum of squares over $\Q$ whenever
$f$ is reducible over $\Q$.
Hence we can assume that $f$ is irreducible over $\Q$, but reducible
over $\C$. So $f$ is either the $K/\Q$-norm of a quadratic form $p\in
K[\x]$ defined over a quadratic field $K/\Q$, or the $K/\Q$-norm of a
linear form $l\in K[\x]$ defined over a field $K$ of degree~$4$. In
either case, $K$ is generated by the coefficients of $f$.
First consider the case $[K:\Q]=2$. It is clear that $f$ is a sum of
squares over $\Q$ when $K$ is imaginary. When $K$ is real, both $p$
and its $K/\Q$-conjugate $p'$ must be nonnegative, since $f=pp'$ is
nonnegative. Hence $p$ is nonnegative with respect to every (real)
place of $K$, and therefore $p$ is a sum of squares over $K$, being a
quadratic form. Now Hillar's result \cite{hi} implies that $f$ is a
sum of squares over $\Q$.
\end{lab}

\begin{lab}
It remains to consider the case when $f=N_{K/\Q}(l)$ where $[K:\Q]=
4$ and $l\in K[\x]$ is a linear form whose coefficients generate $K$.
When $K/\Q$ has a quadratic subfield $L/\Q$, we can write $f=N_{L/\Q}
(N_{K/L}(l))$ and conclude that $f$ is a sum of squares over $\Q$, by
the argument in \ref{reduc1}. So we can assume that $K/\Q$ has no
proper intermediate field. This means that $\Gal(\ol\Q/\Q)$ acts on
$\Hom(K,\ol\Q)$ as the alternating or symmetric group.
Let $l_i=0$ ($i=1,2,3,4$) be the four Galois conjugates of the line
$l=0$. When $l_1,\dots,l_4$ fail to be in general position, all four
meet in a common $\Q$-point. After a suitable coordinate change we
are then in the case of binary forms, in which it is clear that $f$
is a sum of squares over $\Q$.
The proof of Theorem \ref{terquaronly} is complete.
\qed
\end{lab}


\section{Some open questions}

Here are several natural questions that arise in connection with the
results of this paper. Let always $\x=(x_0,\dots,x_n)$.

\begin{lab}
In Theorem \ref{stq} we constructed forms in $\Q[\x]$ that are sums
of squares of forms over $\R$, but not over $\Q$. All our examples
split over $\C$ as products of linear forms. Are there examples that
are irreducible over $\C$? Are there examples that are strictly
positive definite, i.e., that have no nontrivial real zeros? Are
there examples that define a nonsingular projective hypersurface?
(The last question is a common sharpening of the former two.)
\end{lab}

\begin{lab}
Let $K$ be a real number field, and let $f$ be a form in $\Q[\x]$
that is a sum of squares of forms over $K$. When $K$ is totally real,
it follows that $f$ is a sum of squares over $\Q$ (Hillar \cite{hi},
c.f.\ also Section \ref{sect:totreal}). Are there other sufficient
conditions on $K$ that allow the same conclusion?
\end{lab}

\begin{lab}\Label{oddsos}%
More specifically, let $K$ be a number field of odd degree, and
assume that a form $f$ over $\Q$ is a sum of squares over $K$. Then,
is $f$ a sum of squares over $\Q$?
\end{lab}

\begin{lab}
We may generalize the last question to arbitrary linear matrix
inequalities. Thus, let $A_0,\dots,A_r$ be symmetric matrices of some
size with rational coefficients, and assume that there exists $x=
(x_1,\dots,x_r)\in K^r$ such that the matrix $A(x):=A_0+\sum_{i=1}^r
x_iA_i$ is positive semidefinite with respect to every real place of
$K$. If $[K:\Q]$ is odd, does there exist $x\in\Q^r$ such that $A(x)$
is positive semidefinite? For $r=1$, the answer is yes.
\end{lab}

\begin{lab}
When $f\in\Q[\x]$ is any nonnegative form, there exists a sum of
squares $h\ne0$ of forms in $\Q[\x]$ such that $fh$ is a sum of
squares of forms in $\Q[\x]$. Assuming that $f$ is a sum of squares
of forms over $\R$, can we give an upper bound to $\deg(h)$, for
example in terms of $n$ and $d=\deg(f)$?

Note that there is one case in which the results of this paper give
an answer to this question, namely $(n,d)=(2,4)$. Here $\deg(h)=2$
suffices by \ref{terquaronly} and \ref{denomh17}.
\end{lab}


\end{document}